\documentclass[reqno]{amsart}

\usepackage{amsmath,amssymb,amscd, accents} \usepackage{graphicx}
\DeclareGraphicsExtensions{.eps} \usepackage{mathrsfs}
\usepackage[mathcal]{eucal}

\newtheorem{Thm}{Theorem}{\bfseries}{\itshape}
\newtheorem*{Thm*}{Theorem}{\bfseries}{\itshape}
\newtheorem{Cor}{Corollary}{\bfseries}{\itshape}
\newtheorem{Prop}[Cor]{Proposition}{\bfseries}{\itshape}
\newtheorem{Lem}[Cor]{Lemma}{\bfseries}{\itshape}
\newtheorem*{Lem*}{Lemma}{\bfseries}{\itshape}
{\bfseries}{\itshape}
{\bfseries}{\itshape}
\newtheorem{Pblm}{Problem}{\bfseries}{\itshape}
\newtheorem{Def}[Cor]{Definition}{\bfseries}{\rmfamily}
{\scshape}{\rmfamily}
\newtheorem{Rem}[Cor]{Remark}{\scshape}{\rmfamily}

\renewcommand\ge{\geqslant} \renewcommand\le{\leqslant}
\let\tildeaccent=\~ \let\hataccent=\^
\renewcommand\~[1]{\widetilde{#1}} \renewcommand\^[1]{\widehat{#1}}

\def\<{\left<} \def\>{\right>} \def\({\left(} \def\){\right)}
\def\abs#1{\left\vert #1 \right\vert} \def\norm#1{\left\Vert #1
  \right\Vert}

\let\parasymbol=\S \def\secref#1{\parasymbol\ref{#1}}

\let\polishL=l \def\Zoladek.{\.Zol\c adek}

 \def\Im{\operatorname{Im}}
 
 \def\ord{\operatorname{ord}}
 \def\etc.{\emph{etc}.}

\def\:{\colon} \def\R{{\mathbb R}} \def\C{{\mathbb C}}  \def\N{{\mathbb N}}  
 \def\e{\varepsilon} \def\S{\varSigma}

\def\res{\operatornamewithlimits{Res}} \def\d{{\mathrm d}}

 \let\PolishL=\L \def\Lojas.{\PolishL ojasiewicz}

\def\cL{{\mathcal L}}

\def\rest#1{{\vert_{#1}}}

 \def\deg{\operatorname{deg}}

\def\Log{\operatorname{Log}} \def\length{\operatorname{length}}
\def\CP{{\C P}} \def\dL{{\d\Log}}

 \def\L{\mathcal{L}} 
\newcommand\mmk[1][k]{\widetilde{m}_{#1}}
\newcommand\mk[1][k]{m_{#1}}

\def\degP{d_P}
\def\degQ{d_Q}
\def\degp{d_p}
\def\degq{d_q}

\begin{document}
\setlength{\parindent}{0pt} \setlength{\parskip}{0.5em}

\title[Bounds for cyclicity of zero solution of Abel equation]{Uniform
  upper bounds for the cyclicity of the zero solution of the Abel
  differential equation}

\author{Dmitry Batenkov}\address{Department of Computer Science,
  Technion - Israel Institute of Technology, Haifa 32000,
  Israel}\email{batenkov@cs.technion.ac.il} \author{Gal
  Binyamini}\address{Department of Mathematics, University of Toronto, Toronto,
  Canada}\email{galbin@gmail.com} \thanks{G.B. was supported by the
  Banting Postdoctoral Fellowship and the Rothschild Fellowship}

\begin{abstract}
  Given two polynomials $P,q$ we consider the following question:
  ``how large can the index of the first non-zero moment
  $\~m_k=\int_a^b P^k q$ be, assuming the sequence is not identically
  zero?''. The answer $K$ to this question is known as the moment
  Bautin index, and we provide the first general upper bound:
  $K\le2+\deg q+3(\deg P-1)^2$. The proof is based on qualitative
  analysis of linear ODEs, applied to Cauchy-type integrals of
  certain algebraic functions.

  The moment Bautin index plays an important role in the study of
  bifurcations of periodic solution in the polynomial Abel equation
  $y'=py^2+\e qy^3$ for $p,q$ polynomials and $\e\ll1$. In particular,
  our result implies that for $p$ satisfying a well-known generic
  condition, the number of periodic solutions near the zero solution
  does not exceed $5+\deg q+3\deg^2 p$. This is the first such bound
  depending solely on the degrees of the Abel equation.
\end{abstract}
\maketitle
\date{\today}

\section{Introduction}

\subsection{Polynomial moments and their Bautin index}
\label{sec:intro}

Throughout this paper $P,Q\in\C[z]$ will denote a pair of polynomials,
and $p,q$ their respective derivatives. We will denote the degrees of
$P,Q$ (resp. $p,q$) by $\degP,\degQ$ (resp. $\degp,\degq$). We also
fix two points $a,b\in\C$.

Two related types of moment sequences corresponding to this data have
been considered in the literature,
\begin{align}
  &\mk = \mk(P,Q) := \int_a^b P^k(z) Q(z) p(z) \d z, &\quad k=0,1,2,\dots \label{eq:moments} \\
  &\mmk = \mmk(P,q):=\int_a^b P^k(z)q(z)\d z, &\quad k=0,1,2,\dots \label{eq:tilde-moments}
\end{align}
These two sequences are related through a simple formula due to
\cite{PRY} (see~\secref{sec:moment-gen} for details), and may be
viewed essentially as different normalizations. The former
normalization is particularly convenient for the study of moment
generating functions, as explained in~\secref{sec:moment-gen}. The
latter appears more directly in the study of perturbations of the Abel
equation, as explained in~\secref{sec:abel}.

\begin{Def}
  We define the \emph{vanishing index} $N(P,Q,a,b)$ to be the first
  index $k$ such that $\mk(P,Q)\neq0$, or $\infty$ if no such $k$
  exists. We define the \emph{moment Bautin index}
  $N(\degP,\degQ,a,b)$ to be
  \begin{equation}
    N(\degP,\degQ,a,b) := \sup_{\substack{ \deg Q\le\degQ,\deg P\le\degP \\ N(P,Q,a,b)<\infty}} N(P,Q,a,b)+1,
  \end{equation}
  i.e. the least $k\in\N$ with the property that $N(P,Q,a,b)\ge k$
  implies $N(P,Q,a,b)=\infty$ for any $P,Q$ with $\deg P\le\degP$ and
  $\deg Q\le\degQ$.

  We define $\~N(P,q,a,b)$ and $\~N(\degP,\degq,a,b)$ analogously.
\end{Def}

\begin{Rem}
  The moments $\mk(P,Q)$ are polynomials in the coefficients of $P,Q$.
  Let $R$ denote the ring of polynomials in these coefficients and
  $I_k\subset R$ denote the ideal by $m_0,\ldots,m_k$. Then
  $N(\degP,\degQ,a,b)$ defined above is the first index for which the
  chain $\{\sqrt{I_k}\}_{k\in\N}$ stabilizes. In particular, from
  noetherianity it follows that this index is well-defined (finite).
  An analogous remark holds for $\~N(\degP,\degq,a,b)$.
\end{Rem}

The moment Bautin index has been studied in various special cases,
motivated primarily by its relation to perturbations of the Abel
equation (see~\secref{sec:abel} for an overview). Bounds have been
obtained in various special cases, including the cases $\degP=2,3$. We
refer the reader to \cite{blinov_local_2005, briskin_tangential_2005}
and references therein for details. However, to our knowledge no
general bound has been available. Our main result is the following
general bound for the moment Bautin index.

\begin{Thm}\label{thm:main-thm}
  For any $\degP,\degQ\in\N$ we have
  \begin{equation}\label{eq:main-bound}
    N(\degP,\degQ,a,b) \le \degQ+3(\degP-1)^2.
  \end{equation}
  Similarly, for any $\degP,\degq$ we have
  \begin{equation}
    \~N(\degP,\degq,a,b) \le 2+\degq+3(\degP-1)^2.
  \end{equation}
\end{Thm}

It is shown in \secref{sec:moment-gen} (following \cite{PRY}) that the
second bound in Theorem~\ref{thm:main-thm} follows immediately from
the first. Our approach to the proof of the first bound is based on
the following two observations:
\begin{enumerate}
\item The vanishing index $N(P,Q,a,b)$ is essentially equivalent to
  the order of the zero at $t=\infty$ of the moment generating
  function $H(t)$ for the moment sequence $\{\mk\}$. It turns out
  \cite{PRY} that $H(t)$ admits an analytic expression as a Cauchy
  type integral for the algebraic function
  $Q(P^{-1}(z))$. 
\item The Cauchy type integral above satisfies a (non-homogeneous)
  linear differential equation of Fuchsian type \cite{kisunko}.
\end{enumerate}
The problem of estimating $N(P,Q,a,b)$ is thus reduced to the study of
the order of zero at $t=\infty$ of solutions of certain Fuchsian
differential equations. A detailed analysis of the Fuchsian
differential operator involved, and elementary considerations
concerning its monodromy, allow us to give an a-priori upper bound for
this order of zero, thus proving Theorem~\ref{thm:main-thm}.

\subsection{Perturbations of the Abel equation}
\label{sec:abel}

The classical Hilbert's 16th problem asks for bounding the number of
limit cycles, i.e. isolated closed trajectories, of the polynomial
vector field
\begin{align}\begin{split}\label{eq:poly-field}
    \frac{\d x}{\d t} &= -y+F(x,y), \\
    \frac{\d y}{\d t} &= x+G(x,y).\end{split}
\end{align}
The closely related \emph{Poincar\'e Center-Focus Problem} asks for
explicit conditions on the polynomials $F,G$ in order for the system
\eqref{eq:poly-field} to have a center. These problems remain widely
open, although during the years many partial results have been
obtained (see \cite{IY:book} for an exposition).

An alternative context for the study of the problems above is provided
by the Abel differential equation,
\begin{equation}\label{eq:abel}
  y'=p(x)y^2 + q(x) y^3, \quad x\in[a,b] \subset \R,
\end{equation}
where $p,q$ can be polynomials, trigonometric polynomials or even
analytic functions \cite{neto_number_1980}. A periodic solution in
this context corresponds to solution $y(x)$ satisfying $y(a)=y(b)$,
and a center corresponds to an Abel equation where every solution with
a sufficiently small initial condition is periodic. The Abel equation analogue of the Hilbert 16th problem, known as the Smale-Pugh problem, is to bound the number of periodic solutions of \eqref{eq:abel} in terms of the degrees of $p$ and $q$. It is generally
believed that some (but not all) of the essential difficulties in the
study of~\eqref{eq:poly-field} can be observed in~\eqref{eq:abel},
even when one restricts to the case of polynomial
coefficients. 
On the other hand, the polynomial Abel equation allows for several
important technical simplifications, and significant progress has been
achieved for the Center-Focus in this context using tools from
polynomial composition algebra and algebraic geometry
\cite{briskin_center_2010, pakovich_solution_2014}.

The Smale-Pugh problem for the polynomial Abel equation remains open.
Its infinitesimal version, first suggested in
\cite{briskin_tangential_2005}, is as follows:
\begin{Pblm}\label{prb:inf-SP}
  How many periodic solutions can a small perturbation
  \begin{equation}\label{eq:abel-inf}
    y' = p(x) y^2 + \varepsilon q(x) y^3,\quad x\in\left[a,b\right]
  \end{equation}
  of the ``integrable'' equation $y'=p(x)y^2$ have?
\end{Pblm}
This is an Abel equation analog of the ``Infinitesimal Hilbert 16th
problem'' for which an explicit bound was obtained in
\cite{binyamini_number_2010}. Following \cite{briskin_tangential_2005}, in this paper we focus our attention on the periodic solutions bifurcating from \emph{the zero solution} of \eqref{eq:abel-inf}.

The unperturbed equation ($\e=0$) is a center if and only if
$\int_a^bp(x)\d x=0$. Thus we may choose the primitive $P$ such that
$P(a)=P(b)=0$. As in the classical case, the study of the bifurcation
of periodic solutions as well as the center conditions for the
perturbation~\eqref{eq:abel-inf} begins with the study of the first
variation of the Poincaré map.

For technical reasons it is customary to consider the ``reverse'' map
from time $x=b$ to time $x=a$. Namely, let $G(y):(\C,0)\to(\C,0)$
denote the germ of the analytic map assigning to each initial
condition $y_b$ the value $G(y_b)=\eta(a)$, where $\eta$ is a solution
of~\eqref{eq:abel-inf} satisfying $\eta(b)=y_b$. We may view $G$ as a
germ of an analytic function in the coefficients of the polynomials
$p,q$ and $\e$ as well. Fixed points of $G$ correspond to periodic
solutions, and the identical vanishing of $G(y)$ corresponds to a
center. An explicit computation
\cite[Proposition~4.1]{briskin_tangential_2005} gives the expansion
\begin{equation}\label{eq:first-var}
  \frac{\d}{\d\e}\big|_{\e=0} G(y) = -y^3 \int_a^b \frac{q(x)}{1-yP(x)}\d x = \sum_{k=0}^\infty \mmk y^{k+3}.
\end{equation}
 
As in the classical study of perturbation of Hamiltonian planar
systems, it follows from this variational computation that the number
of periodic solutions bifurcating from the zero solution
of~\eqref{eq:abel-inf} is bounded by the order of zero of the right
hand side, i.e. $\~N(P,q,a,b)+3$, assuming that this number is finite.
On the other hand, if the first variation vanishes identically then
one must in general consider higher order variations in $\e$,
further complicating the study of bifurcating periodic solutions.

A surprising feature of the Abel equation~\eqref{eq:abel-inf} is that
for many polynomials $p$, the vanishing of the first
variation~\eqref{eq:first-var} automatically implies the identical
vanishing of the Poincaré map. Toward this end we recall the following
definition.

\begin{Def}[\cite{blinov_local_2005}]
  The polynomials $P,Q$ are said to satisfy the composition condition
  (PCC) on $[a,b]$ if there exists a polynomial $W(x)$ with
  $W(a)=W(b)$, and polynomials $\tilde{P},\tilde{Q}$ such that
  \[
  P(x)=\tilde{P}(W(x)),\qquad Q(x)=\tilde{Q}(W(x)).
  \]

  A polynomial $P$ is called ``definite'' (w.r.t $a,b$), if for any
  polynomial $Q$, vanishing of all the moments $\mmk(P ,q)$ implies
  PCC for $P,Q$.
\end{Def}
Definite polynomials are ubiquitous. In the deep works
\cite{pakovich_solution_2007,pakovich_solution_2014} all
counter-examples have been classified and shown to admit a rigid
algebraic structure.

Whenever the polynomials $P,Q$ satisfy the PCC, the corresponding Abel
equation~\eqref{eq:abel} automatically admits a center, as can be seen
by a simple change of variable argument. We thus see that for a
definite polynomial $P$, the vanishing of all moments $\mmk(P,q)$
implies the \emph{identical} vanishing of the Poincaré map $G(y)$.
Therefore, in a sense the bifurcation of periodic solutions
in~\eqref{eq:abel-inf} is fully controlled by the first
variation~\eqref{eq:first-var}. More formally, the following holds.

\begin{Thm}[\cite{blinov_local_2005}]\label{thm:bautin-index-and-cyclicity}
  Let $P$ be a definite polynomial, and fix the parameters $a,b,d_q$.
  Then for any $\norm{q}\ll1$ with $\deg q\le d_q$, the number of periodic
  solutions of~\eqref{eq:abel} with $|y(a)|\ll 1$ is at most $\~N(\degP,\degq,a,b)+3$.
\end{Thm}

As a Corollary of Theorem~\ref{thm:main-thm} we therefore have the
following first general estimate for the number of limit cycles near
the zero solution for an Abel equation \eqref{eq:abel} with $\norm{q}$
small.

\begin{Cor}\label{cor:main}
  Under the conditions of
  Theorem~\ref{thm:bautin-index-and-cyclicity}, the number of periodic
  solutions is bounded by $5+\degq+3\degp^2$.
\end{Cor}

\subsection{Organization of the paper}

In~\secref{sec:moment-gen} we introduce moment generating functions
for the two moment sequences $\{\mk\},\{\mmk\}$ which turn out to be
Cauchy-type integrals. We thus reduced the study of the corresponding
vanishing indices to the study of the order of zero of these
generating functions at infinity. In~\secref{sec:cauchy-int} we give a
slightly generalized version of the result of \cite{kisunko} which
states that if a function $g(z)$ satisfies a linear ODE $\cL g=0$ then
the corresponding Cauchy-type integral $I(t)$ satisfies a
non-homogeneous linear ODE $\cL I=R$, where $R$ is a rational function
of known degree. Subsequently, in~\secref{sec:diff-op} we explicitly
derive the corresponding non-homogeneous ODE for the moment generating
functions. Finally in~\secref{sec:order-zero} we produce estimates for
the order of zero the moment generating function at infinity using
qualitative methods of linear ODEs.

\section{Polynomial moments and generating functions}
\label{sec:moment-gen}

Recall the notations of~\secref{sec:intro}. We introduce moment
generating functions with the corresponding integral expression for
the sequences $\{\mk\},\{\mmk\}$ as follows:
\begin{align}
  H(t) = \sum_{k=0}^\infty m_k t^{-(k+1)} \qquad H(t) = \int_a^b \frac{Q(z) p(z)}{t-P(z)}\d z, \label{eq:mk-gen} \\
  \~H(t) = \sum_{k=0}^\infty \~m_k t^{-(k+1)} \qquad \~H(t) = \int_a^b \frac{q(z)}{t-P(z)}\d z. \label{eq:mmk-gen}
\end{align}
Clearly,
\begin{equation}
  \ord_\infty H(t) = N(P,Q,a,b)+1 \qquad \ord_\infty \~H(t) = \~N(P,q,a,b)+1.
\end{equation}
In particular, we have the following.
\begin{Prop}\label{prop:reduction-to-order-of-I}
  We have
  \begin{equation}
    N(\degP,\degQ,a,b) = \sup_{H(t)\not\equiv 0} \ord_{t=\infty} H(t).
  \end{equation}
  where the supremum is taken over all pairs $P,Q$ with respective
  degrees bounded by $\degP,\degQ$ and $H(t)$ denotes the
  corresponding moment generating function.
\end{Prop}

It turns out that $H(t)$ and $\~H(t)$ are related by a simple formula,
which implies in particular that the study of their orders of
vanishing at $t=\infty$ are essentially the same
\cite[Claim,~p.40]{PRY}. We repeat the argument of \cite{PRY} in order
to obtain an explicit description of relation between these orders.

\begin{Lem}\label{lem:mk-vs-mmk}
  The condition $\~m_0=0$ is equivalent to $Q(a)=Q(b)$. Moreover,
  under this condition we have $\~m_{k+1}=-(k+1)m_k$ for $k\in\N$.
  In particular, we have
  \begin{equation}
    \~N(P,q,a,b) \le N(P,Q,a,b)+1.
  \end{equation}
\end{Lem}
\begin{proof}
  Derivating under the integral sign we have
  \begin{multline}
    \frac{\d H(t)}{\d t} = -\int_a^b \frac{Q(z) p(z)}{(t-P(z))^2}\d z
    = -\int_a^b Q \d\(\frac{1}{t-P(z)}\)  \\
    = -\big[\frac{Q(z)}{t-P(z)}\big]_a^b + \int_a^b
    \frac{q(z)}{t-P(z)}\d z = \frac{Q(a)}{t-P(a)} -
    \frac{Q(b)}{t-P(b)} + \~H(t)
  \end{multline}
  Comparing the $t^{-1}$ coefficient we see that $\~m_0=0$ if and only
  if $Q(a)=Q(b)$, and under this condition $\~m_{k+1}=-(k+1)m_k$ as
  claimed.
\end{proof}

The moment generating function~\eqref{eq:mk-gen} has the form of a
Cauchy integral. Indeed, choose the curve of integration $\gamma'$
from $a$ to $b$ in~\eqref{eq:mk-gen} to be some smooth curve avoiding
the critical values of $P(z)$ (except perhaps at the endpoints). Then
setting $\gamma=P(\gamma')$ and substituting $w=P(z)$
in~\eqref{eq:mk-gen} we obtain
\begin{equation} \label{eq:mk-gen-cauchy}
  H(t) = \int_\gamma \frac{Q(P^{-1}(w))}{t-w} \d w 
\end{equation}
where $P^{-1}(w)$ denotes the branch of $P^{-1}$ lifting $\gamma$ to $\gamma'$.

\section{Cauchy-type integrals and linear differential
  operators}
\label{sec:cauchy-int}

Let $\cL$ be a scalar differential operator,
\begin{equation}
  \cL = c_r(z)\partial^r+\cdots+c_0(z), \qquad c_0,\ldots,c_r\in\C[z].
\end{equation}
Let $\gamma\subset\C$ be a smooth curve, and assume that $\gamma$ does
not pass through the singular points of $\cL$, except perhaps at its
endpoints. Finally let $g$ be a solution of $\cL g=0$ defined on
$\gamma$, and assume further that $g$ is bounded on $\gamma$
(including at the possibly singular endpoints). We denote by $p_+,p_-$
the endpoints of $\gamma$.

Then we define the Cauchy-type integral
\begin{equation}
  I(t) = \int_\gamma \frac{g(z)}{z-t} \d z
\end{equation}
It is classically known that $I(t)$ is a holomorphic functions defined
on $\C\setminus\gamma$, and moreover that the boundary values $I^+$
and $I^-$ of $I(t)$ on $\gamma$ from above and below respectively
satisfy $I^+-I^-=g\rest\gamma$. Moreover $I(t)$ can be analytically
continued along any path avoiding the endpoints of $\gamma$.

Kisunko \cite{kisunko} proved the following (under the extra mild
assumption that $g$ is holomorphic at the endpoints of $\gamma$).

\begin{Prop} \label{prop:kisunko} We have $\cL I(t)=R(t)$ where $R(t)$
  is a rational function having poles of order at most $r$ at
  $p_+,p_-$ and no other poles on $\C$.
\end{Prop}
\begin{proof}[Sketch of proof.]
  By the classical properties of $I(t)$ mentioned above, $\cL I(t)$ is
  a (possibly multivalued) analytic function on
  $\C\setminus\{p_+,p_-\}$ with ramifications $p_+,p_-$, and the
  difference between the two branches near the branch cut at $\gamma$
  is $g$. But since $\cL g=0$, the boundary values of $\cL I^+$ and
  $\cL I^-$ agree, so $\cL I$ is in fact a univalued holomorphic
  function defined on $\C\setminus\{p_+,p_-\}$. We will show that it
  has poles of order at most $r$ at $p_+,p_-$ and at most a pole at
  $\infty$.

  Since $g\rest\gamma$ is bounded, we may derive under the integral
  sign and write
  \begin{equation} \label{eq:LI-estimate} \cL I(t) = \sum_{k=0}^r
    \frac{(-1)^k}{k!} c_k(t) \int_\gamma \frac{g(z)}{(z-t)^{k+1}}\d z
  \end{equation}
  We now show that $\cL I(t)$ admits polynomial growth of order at
  most $r$ at $p_+$ (and the same arguments work for $p_-$). It is
  enough to consider each of the integrals in~\eqref{eq:LI-estimate}
  separately. Moreover, we may assume that $\gamma$ is a small piece
  of a smooth curve near $p_+$ (because the integral over the rest of
  $\gamma$ is analytic at $p_+$). Choose a coordinate system where
  $p_+=0$.

  Let $M$ denote an upper bound for $\abs{g(z)}$ on $\gamma$. Let $t$
  be a point in a punctured neighborhood of $p_+$. Since $g(z)$ admits
  analytic continuation along any curve in the punctured neighborhood,
  may deform $\gamma$ without changing $\cL I(t)$ so that for some
  positive constants $C,D$ independent of $t$,
  \begin{enumerate}
  \item For every $z\in\gamma$, we have $|z-t|\ge C|t|$ and also
    $|z-t|>C|z|$.
  \item On $\gamma$ we have $|\d z|\le D \d|z|$.
  \item Write $g=\gamma_1+\gamma_2$ where $\gamma_1$ is the part of
    $\gamma$ which lies in $\{z:|z|<2t\}$ and $\gamma_2$ is the rest.
    Then the length of $\gamma_1$ is at most $D|t|$, and the length of
    $\gamma_2$ is at most $D$.
  \end{enumerate}
  We now estimate
  \begin{multline}
    \big|\int_\gamma \frac{g(z)}{(z-t)^{k+1}}\d z\big|
    \le \int_{\gamma_1} M (C |t|)^{-k-1}\abs{\d z} + \int_{\gamma_2} \frac{MD}{(C|z|)^{k+1}} \d|z| \\
    \le \length(\gamma_1) M (C |t|)^{-k-1}+ \left[
      -\frac{MDC^{k+1}}{k} |z|^{-k} \right]^{\cdots}_{2|t|} \le
    O(|t|^{-k})
  \end{multline}
  proving the claim.
  
  Finally, it is easy to see that $I(t)$ and its derivatives have a
  zero at $t=\infty$, and since the coefficients of $\cL$ are
  polynomial it follows that $I(t)$ has at most a pole at $\infty$ as
  well.
\end{proof}

\section{A differential operator for $Q(P^{-1})$}
\label{sec:diff-op}

Let $V$ denote the linear space spanned by the $\degP$ branches of the
algebraic function $g(z):=Q(P^{-1}(z))$. We denote $r:=\dim V$, and
note that $r$ may be strictly smaller than $\degP$. Denote by
$p_1,\ldots,p_s$ the critical values of $P$.

\subsection{The operator $\cL$}
By a theorem of Riemann \cite[Theorem~19.7]{IY:book}, there exists a
linear $r$-th order differential operator $\cL$, with polynomial
coefficients,
\begin{equation} \label{eq:cl-def} \cL =
  c_r(z)\partial^r+\cdots+c_0(z), \qquad c_0,\ldots,c_r\in\C[z]
\end{equation}
whose solution space coincides with $V$. Moreover, $\cL$ is uniquely
determined by the requirement that $c_r,\ldots,c_0$ do not share a
non-trivial common factor. We recall the construction of $\cL$.

Recall that the Wronskian $W(f_1,\ldots,f_n)$ of a tuple of functions
is defined to be
\begin{equation}
  W(f_1,\ldots,f_n) := \det
  \begin{pmatrix}
    f_1 & \cdots & f_n \\
    \partial f_1 & \cdots & \partial f_n \\
    & \vdots & \\
    \partial^{n-1} f_1 & \cdots & \partial^{n-1} f_n
  \end{pmatrix}
\end{equation}
Now let $g_1,\ldots,g_r$ denote $r$ branches of $g(z)$ which span $V$.
Then clearly for any $f\in V$ we have $W(g_1,\ldots,g_r,f)=0$. We
define the operator $\~L$ given by
\begin{equation} \label{eq:Ltilde} \~\cL(f) =
  \frac{W(g_1,\ldots,g_r,f)}{W_r}= \big[\partial^r+\sum_{k=0}^{r-1}
  \~c_k(z) \partial^k\big] f \qquad \text{where }
  \~c_k=\frac{W_k(g_1,\ldots,g_r)}{W_r(g_1,\ldots,g_r)}
\end{equation}
where $W_i$ are the minors obtained when expanding the Wronskian
$W(g_1,\ldots,g_r,f)$ along the last column. If the monodromy of $g$
along a closed curve $\gamma$ induces the linear automorphism
$M_\gamma:V\to V$ then the corresponding monodromy along $\gamma$ of
each $W_k$ is given by multiplication by $\det M_\gamma$. In
particular, the coefficients $\~c_k$ are univalued functions.

\subsection{The divisors $[W_k]$}

Let $k=0,\ldots,r$ and $z_0\in\CP$. Choose any local representative of
the functions $g_1,\ldots,g_r$. Since these functions have at most a
finite ramification and moderate growth at $z_0$, we may expand
\begin{equation}
  W_k = \sum_{j=-N}^\infty a_{k,j}(z-z_0)^{j/q}
\end{equation}
where $q$ and $N$ are some natural numbers. Suppose that $a_{k,j_0}$
is the first non-zero coefficient among the $a_{k,j}$. Then we say
that the \emph{fractional order of $W_k$ at $z_0$} is
$\ord_{z_0} W_k:=j_0/q$. This notion is well-defined: indeed, the
monodromy of $W_k$ along any curve is given by multiplication by a
non-zero constant and hence does not change the order. We define the
fractional divisor $[W_k]$ of $W$ to be
\begin{equation}
  [W_k]:=\sum_{z\in\CP} \ord_z W_k(z) [z].
\end{equation}
This sum is locally-finite, and hence finite. Moreover it is clear
that $[\~c_k]=[W_k]-[W_r]$. In particular, $\~c_k$ admits finitely
many singularities of finite order. Since we have already seen that
$\~c_k$ is univalued, it is in fact a rational function.

We can also write the divisor $[W_k]$ in terms of residues. Indeed,
since the monodromy of $W_k$ along any curve is given by
multiplication by a constant, the one-form $\dL W_k$ is a univalued
one-form. It is easy to verify in local coordinates that it in fact
has only finitely many poles, all of first order, and

\begin{equation}
  [W_k] = \sum_{z\in\CP} \res_z (\dL W_k) [z].
\end{equation}
For any divisor $D=\sum n_i [z_i]$ we denote $D_{z_i}=n_i$ and
\begin{equation}
  \deg D = \sum n_i, \qquad D^+ = \sum_{n_i\ge0} n_i [z_i], \quad D^- = -\sum_{n_i\le0} n_i [z_i].
\end{equation}
In particular, it follows from the above that $\deg [W_k]=0$.

\subsection{An estimate for $\deg[W_r]^+$}

Our next goal is to estimate $\deg[W_r]^+$. Since $[W_r]$ is
principal, it will suffice to estimate $\deg[W_r]^-$. Recall that
$W_r=W(g_1,\ldots,g_r)$ where $g_k=Q(P^{-1}_k(z))$ and
$P^{-1}_1(z),\ldots,P^{-1}_r(z)$ denote $r$ different branches of
$P^{-1}(z)$. If $z\in\C$ is not a critical value of $P$ then these
functions are all holomorphic around $z$, and hence $[W_r]_z$ is
non-negative.

Let the critical value $p_i$ have exactly $m_i<\degP$ preimages, and
write $b_i:=\degP-m_i$ for the number of critical points (counted with
multiplicities) over $p_i$. Then at most $2b_i$ of the branches $g_k$
may be ramified at $p_i$. We expand the determinant defining $W_r$ and
note that:
\begin{itemize}
\item since $g_k$ is bounded, its order is non-negative;
\item differentiation can decrease the order by at most $1$;
\item differentiation cannot decrease the order below zero for
  holomorphic $g_k$.
\end{itemize}
We thus conclude that
\begin{equation*}
  \ord_{p_i} W_r > (-r+1)+\cdots+(-r+\nu), \qquad \text{where }\nu=\min(r,2b_i).
\end{equation*}
Since $b_1+\cdots+b_s=\degP-1$, it is not hard to see that the maximal
value for the following sum is obtained when $b_i=1$ for
$i=1,\ldots,s$, and in any case
\begin{equation}
  \sum_{i=1}^s \ord_{p_i} W_r > - (2\degP-3)(\degP-1).
\end{equation}

It remains to estimate the order of $W_r$ at $\infty$. Choose a
coordinate $w$ around $\infty$ such that $P(w)=w^{-\degP}$. Then any
branch of $Q(P^{-1}(w))$ has the Puiseux expansion
\begin{equation}
  Q(P^{-1}(w)) = Q(w^{-1/\degP}) = \alpha w^{-\degQ/\degP} + \cdots, \qquad \alpha\neq0
\end{equation}
where $\cdots$ denote higher order terms. Moreover, the derivative
$\partial_z=-w^2\partial_w$ increases the order of zero at $w=0$ by at
least one. Expanding the determinant defining $W_r$ we see that
\begin{equation}
  \ord_\infty W_r \ge -\frac{r\degQ}{\degP} + \frac{r(r-1)}{2}.
\end{equation}
In conclusion, we have
\begin{equation} \label{eq:w0-deg} \deg [W_r]^+ = \deg [W_r]^- \le
  \frac{\degQ r}{\degP} + (2\degP-3)(\degP-1)-\frac{r(r-1)}{2}.
\end{equation}

\subsection{An estimate for $\deg c_r$}

We wish to derive an estimate for the number of singularities of
$\cL$, or more specifically for $\deg c_r$. By definition, $c_r$ is a
polynomial and $[c_r]^+$ is the least common upper bound for
$[\~c_0]^-,\ldots,[\~c_{r-1}]^-$ in~\eqref{eq:Ltilde}. Recall that
$[\~c_k]=[W_k]-[W_r]$.

We first note that $\cL$ is a Fuchsian operator. Indeed, since the
solutions of $\cL$, being algebraic functions, have moderate growth at
each singularity, this follows from a theorem of Fuchs
\cite[Theorem~19.20]{IY:book}. Thus by definition the order of $[c_r]$
at any point $p\in\C$ cannot exceed $r$. We will apply this to the
points $p_1,\ldots,p_s$.

Let now $z\in\C$ and $z\not\in\{p_1,\ldots,p_s\}$. Then the branches
$g_1,\ldots,g_r$ are holomorphic at $z$, so $z$ is not a point of
$[W_0]^-,\ldots,[W_{r-1}]^-$. In other words, $z$ can only be a point
of $[\~c_0]^-,\ldots,[\~c_{r-1}]^-$ if it comes from $[W_r]^+$. Thus
we see that
\begin{equation}
  [c_r]^+ \le \sum_{i=1}^s r [p_i] + [W_r]^+
\end{equation}
Using~\eqref{eq:w0-deg} and noting that $s\le \degP-1$ and $r\le \degP$,
we have the following Proposition.

\begin{Prop} \label{prop:cr-deg} The following estimate holds,
  \begin{equation}
    \deg c_r = \deg [c_r]^+ \le \frac{\degQ r}{\degP} + 3(\degP-1)^2 - \frac{r(r-1)}{2}.
  \end{equation}
\end{Prop}

\section{Estimate for the order of $H(t)$ at infinity}
\label{sec:order-zero}

Recall from~\secref{sec:moment-gen} that $H(t)$ denotes the moment
generating function~\eqref{eq:mk-gen}, which can be represented
(around $t=\infty$) as a Cauchy-type integral~\eqref{eq:mk-gen-cauchy}
of the algebraic function $g(z)=Q(P^{-1}(z))$. Recall
from~\secref{sec:diff-op} that $V$ denotes the linear span of the
branches of $g(z)$ with $r:=\dim V$ and $\cL$ the differential
operator~\eqref{eq:cl-def} satisfying $V=\ker\cL$.

\begin{Prop} \label{prop:clH-nonzero} If $H(t)\not\equiv0$ then
  $\cL H(t)\not\equiv0$.
\end{Prop}
\begin{proof}
  Assume that $\cL H(t)\equiv0$. Then $H(t)\in V$. Moreover, $H(t)$ is
  holomorphic at $t=\infty$, and in particular it is invariant under
  the monodromy around infinity $M_\infty$ and hence also under the
  operator
  \begin{equation}
    T_\infty: V\to V, \qquad T_\infty := \frac{1}{\degP} \sum_{k=0}^{\degP-1} M_\infty^k.
  \end{equation}
  Recall that $g(z)=Q(P^{-1}(z))$ and $P^{-1}(z)$ has cyclic monodromy
  at $\infty$. It follows that the image of $T_\infty$ is
  one-dimensional and spanned by
  \begin{equation}
    \Im T_\infty = \C\{ S \}, \qquad S(t):= \sum_{w:P(w)=t} Q(w).
  \end{equation}
  Moreover, $S(t)$ is a polynomial: for instance, it is has no poles
  on $\C$ and moderate growth at $\infty$. We conclude that $H(t)$ is
  a polynomial. Finally, $H(t)$ has a zero at $t=\infty$ by
  definition, and since it is also a polynomial it follows that
  $H(t)\equiv0$, contradicting the hypothesis.
\end{proof}

Let $D=t\partial_t$ denote the Euler operator, and recall that it also
gives the Euler operator at $t=\infty$ (up to a sign). The following
proposition describes the behavior of $\cL$ around infinity.

\begin{Prop} \label{prop:cl-fuchs-infty} We may write
  \begin{equation} \label{eq:cl-fuchs-infty} \cL(t) = u(t) \^\cL
    \qquad \^\cL := D^r+\^c_{r-1}D^{r-1}+\cdots+\^c_0,
  \end{equation}
  where $\^c_{r-1},\ldots,\^c_0$ are rational functions, holomorphic
  at $t=\infty$, and $u(t)$ is a rational function satisfying
  \begin{equation}
    \ord_\infty u \ge -\biggl[ \frac{\degQ r}{\degP} + 3(\degP-1)^2 - \frac{r(r+1)}{2} \biggr].
  \end{equation}
\end{Prop}
\begin{proof}
  The existence of an expression~\eqref{eq:cl-fuchs-infty} is a direct
  consequence of the fact that $\cL$ is a Fuchsian operator at
  $t=\infty$ (see \cite[Proposition~19.18]{IY:book}). Using
  Proposition~\ref{prop:cr-deg} we have
  \begin{equation}
    \ord_\infty u = r-\deg c_r \ge -\biggl[ \frac{\degQ r}{\degP} + 3(\degP-1)^2 - \frac{r(r+1)}{2} \biggr],
  \end{equation}
  as claimed.
\end{proof}

Finally we have the following estimate.

\begin{Lem}\label{lem:h-ord-bound}
  If $H(t)\not\equiv0$ then
  \begin{equation}
    \ord_\infty H(t) \le \frac{\degQ r}{\degP} + 3(\degP-1)^2 - \frac{r(r-3)}{2}.
  \end{equation}
\end{Lem}
\begin{proof}
  Using Proposition~\ref{prop:kisunko} we have $\cL H(t)=R(t)$, where
  $R(t)$ has at most two poles of order $r$ in $\C$. Moreover, by
  Proposition~\ref{prop:clH-nonzero} $R(t)$ is non-zero. It follows
  that $\ord_\infty R(t)\le 2r$. Using
  Proposition~\ref{prop:cl-fuchs-infty} we have
  \begin{align*}
    \ord_\infty(\^\cL H(t)) &= \ord_\infty R(t) - \ord_\infty u(t) \\
    &\le 2r+\frac{\degQ r}{\degP} + 3(\degP-1)^2 - \frac{r(r+1)}{2} \\
    &\le \frac{\degQ r}{\degP} + 3(\degP-1)^2 - \frac{r(r-3)}{2}.
  \end{align*}
  It remains only to note that the application of $\^\cL$ cannot
  decrease the order of zero, and the claim follows.
\end{proof}

Finally we complete the proof of our main result.

\begin{proof}[Proof of Theorem \ref{thm:main-thm}]
  If $H(t)\not\equiv0$ then by Lemma~\ref{lem:h-ord-bound}
  \begin{equation}
    \ord_\infty H(t) \le \degQ+3(\degP-1)^2,
  \end{equation}
  and the claim for $N(\degP,\degQ,a,b)$ follows by
  Proposition~\ref{prop:reduction-to-order-of-I}. The claim for
  $\~N(\degP,\degq,a,b)$ then follows from Lemma~\ref{lem:mk-vs-mmk},
  noting that $\degQ=\degq+1$.
\end{proof}

\bibliographystyle{plain} \bibliography{nrefs}

\end{document}